\documentclass[amstex,12pt, amssymb]{article}

\usepackage{mathtext}
\usepackage[cp1251]{inputenc}
\usepackage[T2A]{fontenc}
\usepackage[dvips]{graphicx}
\usepackage{amsmath}
\usepackage{amssymb}
\usepackage{amsxtra}
\usepackage{latexsym}
\usepackage{ifthen}

\newtheorem{theorem}{Theorem}[section]
\newtheorem{lemma}{Lemma}[section]
\newtheorem{corollary}{Corollary}[section]
\newtheorem{proposition}{Proposition}[section]

\newtheorem{remark}{Remark}[section]

\numberwithin{equation}{section}

\def\Xint#1{\mathchoice
   {\XXint\displaystyle\textstyle{#1}}%
   {\XXint\textstyle\scriptstyle{#1}}%
   {\XXint\scriptstyle\scriptscriptstyle{#1}}%
   {\XXint\scriptscriptstyle\scriptscriptstyle{#1}}%
   \!\int}
\def\XXint#1#2#3{{\setbox0=\hbox{$#1{#2#3}{\int}$}
     \vcenter{\hbox{$#2#3$}}\kern-.5\wd0}}
\def\dashint{\Xint-}

\def\cc{\setcounter{equation}{0}
\setcounter{figure}{0}\setcounter{table}{0}}

\begin{document}

\markboth{\centerline{Vladimir Ryazanov, Sergei Volkov}}
{\centerline{Prime ends and mappings on Riemann surfaces}}

\author{{V. Ryazanov, S. Volkov}}

\title{{\bf Prime ends and mappings\\ on Riemann
surfaces}}

\maketitle

\large \begin{abstract} It is proved criteria for continuous and
homeomorphic extension to the boundary of mappings with finite
distortion between domains on the Riemann surfaces by prime ends of
Caratheodory.
\end{abstract}

\bigskip
{\bf 2010 Mathematics Subject Classification: Primary   31A05,
31A20, 31A25, 31B25, 35Q15; Se\-con\-da\-ry 30E25, 31C05, 34M50,
35F45}

\large
\cc

\section{Introduction}

The theory of the boundary behavior in the prime ends for the
mappings with finite distortion has been developed in \cite{KPR} for
the plane domains and in \cite{KR} for the spatial domains. The
pointwise boundary behavior of the mappings with finite distortion
in regular domains on Riemann surfaces was recently studied by us in
\cite{RV3}. Moreover, the problem was investigated in regular
domains on the Riemann manifolds for $n\ge 3$ as well as in metric
spaces, see e.g. \cite{ARS} and \cite{Smol}. It is necessary to
mention also that the theory of the boundary behavior of Sobolev's
mappings has significant applications to the boundary value problems
for the Beltrami equations and for analogs of the Laplace equation
in anisotropic and inhomogeneous media, see e.g. \cite{BGMR},
\cite{GR}--\cite{GRY1}, \cite{KPRS}, \cite{KPRS1}, \cite{MRSY},
\cite{P}, \cite{RSSY} and relevant references therein.

For basic definitions and notations, discussions and historic
comments in the mapping theory on the Riemann surfaces, see our
previous papers \cite{RV1}--\cite{RV4}.

\section{Definition of the prime ends and preliminary remarks}

First recall the necessary definitions of some general notions.
Given a topological space $T$, a {\bf path in $T$} is a continuous
map $\gamma: [a,b]\to T.$ Given $A, B$, $C\subseteq T,$
$\Delta(A,B,C)$ denotes a collection of all paths $\gamma$ joining
$A$ and $B$ in $C,$ i.e., $\gamma(a)\in A$, $\gamma(b)\in B$ and
$\gamma(t)\in C$ for all $t\in (a,b).$ In what follows, $|\gamma|$
denotes the {\bf locus} of $\gamma$, i.e. the image $\gamma([a,b])$.

We act similarly to Caratheodory \cite{Car$_2$} under the definition
of the prime ends of domains on a Riemann surface $\Bbb S$, see
Chapter 9 in \cite{CL}. First of all, recall that a continuous
mapping $\sigma: \Bbb I\to \Bbb S$, $\Bbb I=(0,1)$, is called a {\bf
Jordan arc} in $\Bbb S$ if $\sigma(t_1)\neq\sigma(t_2)$ for $t_1\neq
t_2$. We also use the notations $\sigma$, $\overline{\sigma}$ and
$\partial\sigma$ for $\sigma(\Bbb I)$, $\overline{\sigma(\Bbb I)}$
and $\overline{\sigma(\Bbb I)}\setminus\sigma(\Bbb I)$,
correspondingly. A {\bf cross--cut} of a domain $D\subset\Bbb S$ is
either a closed Jordan curve or a Jordan arc $\sigma$ in the domain
$D$ with both ends on $\partial D$ splitting $D$.

A sequence $\sigma_1,\ldots, \sigma_m,\ldots$ of cross-cuts of $D$
is called a {\bf chain} in $D$ if:

(i) $\overline{\sigma_i}\cap\overline{\sigma_j}=\varnothing$ for
every $i\neq j$, $i,j= 1,2,\ldots$;

(ii) $\sigma_{m}$ splits $D$ into 2 domains one of which contains
$\sigma_{m+1}$ and another one $\sigma_{m-1}$ for every $m>1$;

(iii)  $\delta(\sigma_{m})\to0$ as $m\to\infty$.

Here $\delta(E)=\sup\limits_{p_1, p_2\in\Bbb S}\delta(p_1, p_2)$
denotes the diameter of a set $E$ in $\Bbb S$ with respect to an
arbitrary metric $\delta$ in $\Bbb S$ agreed with its topology, see
\cite{RV1}--\cite{RV3}.

Correspondingly to the definition, a chain of cross-cuts $\sigma_m$
generates a sequence of domains $d_m\subset D$ such that $d_1\supset
d_2\supset\ldots\supset d_m\supset\ldots$ and $\, D\cap\partial\,
d_m=\sigma_m$. Two chains of cross-cuts $\{\sigma_m\}$ and
$\{\sigma_k'\}$ are called {\bf equivalent} if, for every
$m=1,2,\ldots$, the domain $d_m$ contains all domains $d_k'$ except
a finite number and, for every $k=1,2,\ldots$, the domain $d_k'$
contains all domains $d_m$ except a finite number, too. A {\bf prime
end} $P$ of the domain $D$ is an equivalence class of chains of
cross-cuts of $D$. Later on, $E_D$ denote the collection of all
prime ends of a domain $D$ and $\overline D_P=D\cup E_D$ is its
completion by prime ends.

Next, we say that a sequence of points $p_l\in D$ is {\bf convergent
to a prime end} $P$ of $D$ if, for a chain of cross--cuts $\{
\sigma_m\}$ in $P$, for every $m=1,2,\ldots$, the domain $d_m$
contains all points $p_l$ except their finite collection. Further,
we say that a sequence of prime ends $P_l$ converge to a prime end
$P$ if, for a chain of cross--cuts $\{ \sigma_m\}$ in $P$, for every
$m=1,2,\ldots$, the domain $d_m$ contains chains of cross--cuts $\{
\sigma_k'\}$ in all prime ends $P_l$ except their finite collection.

Now, let $D$ be a domain in the compactification
$\overline{\mathbb{S}}$ of a Riemann surface $\Bbb S$ by
Kerekjarto-Stoilow, see a discussion in \cite{RV1}--\cite{RV3}. Open
neighborhoods of points in $D$ is induced by the topology of
$\overline{\mathbb{S}}$. A basis of neighborhoods of a prime end $P$
of $D$ can be defined in the following way. Let $d$ be an arbitrary
domain from a chain in $P$. Denote by $d^*$ the union of $d$ and all
prime ends of $D$ having some chains in $d$. Just all such $d^*$
form a basis of open neighborhoods of the prime end $P$. The
corresponding topology on $\overline D_P$ is called the {\bf
topology of prime ends}.

Let $P$ be a prime end of $D$ on a Riemann surface ${{\Bbb S}}$,
$\{\sigma_m\}$ and $\{\sigma_m'\}$ be two chains in $P$, $d_m$ and
$d_m'$ be domains corresponding to $\sigma_m$ and $\sigma_m'$. Then
$$\bigcap\limits_{m=1}\limits^{\infty}\overline{d_m}\subseteq
\bigcap\limits_{m=1}\limits^{\infty}\overline{d_m'}\subset
\bigcap\limits_{m=1}\limits^{\infty}\overline{d_m}\ ,$$ and, thus,
$$\bigcap\limits_{m=1}\limits^{\infty}\overline{d_m}=
\bigcap\limits_{m=1}\limits^{\infty}\overline{d_m'}\ ,$$ i.e. the
set named by a {\bf body of the prime end} $P$
\begin{equation}\label{eq11}
I(P):=\bigcap\limits_{m=1}\limits^{\infty}\overline{d_m}
\end{equation} depends only on $P$ but not on a choice of a chain of
cross--cuts $\{\sigma_m\}$ in $P$.

It is necessary to note also that, for any chain $\{ \sigma_m\}$ in
the prime end $P$,
\begin{equation}\label{eq111}
\Omega\ :=\ \bigcap\limits_{m=1}\limits^{\infty}{d_m}\ =\
\varnothing\ .
\end{equation}
Indeed, every point $p$ in $\Omega$ belongs to $D$. Moreover, some
open neighborhood of $p$ in $D$ should belong to $\Omega$. In the
contrary case each neighborhood of $p$ should have a point in some
$\sigma_m$. However, in view of condition (iii) then $p\in\partial
D$ that should contradict the inclusion $p\in D$. Thus, $\Omega$ is
an open set and if $\Omega$ would be not empty, then the
connectedness of $D$ would be broken because $D=\Omega\cup\Omega^*$
with the open set $\Omega^* := D\setminus I(P)$.

In view of conditions (i) and (ii), we have by (\ref{eq111}) that
$$I(P)=\bigcap\limits_{m=1}\limits^{\infty}(\partial d_m\cap\partial D)=
\partial D\ \cap\ \bigcap\limits_{m=1}\limits^{\infty}\partial
d_m\ .$$ Thus, we obtain the following statement.

\begin{proposition}\label{pr1} {\it For each prime end $P$ of a domain $D$
on a Riemann surface,} \begin{equation}\label{eq1}
I(P)\subseteq\partial D.\end{equation}
\end{proposition}

\begin{remark}\label{rmk1} {\rm If $D$ is a domain in $\overline{\Bbb S}$ with $\partial
D\subset\Bbb S$, then $I(P)$ is a continuum, i.e. it is a connected
compact set, see e.g. I(9.12) in \cite{Wh}, see also I.9.3 in
\cite{Bou}, and $I(P)$ belongs to only one (connected) component
$\Gamma$ of $\partial D$. Hence we say that the component $\Gamma$
is {\bf associated with the prime end} $P$.

\medskip

Moreover, every prime end of $D$ in the case contains a {\bf
convergent chain} $\{ \sigma_m\}$, i.e., that is contracted to a
point $p_0\in\partial D$. Furthermore, each prime end $P$ contains a
{\bf spherical chain} $\{ \sigma_m\}$ lying on circles
$S(p_0,r_m)=\{ p\in\Bbb S :\ \delta(p,p_0)=r_m\}$ with
$p_0\in\partial D$ and $r_m\to0$ as $m\to\infty$. The proof is
perfectly similar to Lemma 1 in \cite{KR} after the replacement of
metrics, see also Theorem  VI.7.1 in \cite{New}, and hence we omit
it. Note by the way that the condition (iii) does not depend in the
case on the choice of the metric $\delta$ agreed with the topology
of $\Bbb S$ because $\partial D$ has a compact neighborhood.}
\end{remark}

\bigskip

It is known that the conformal modulus $M$ of the family of all
paths joining a pair of the opposite sides of a rectangle is equal
to the ratio of lengths of other pair of opposite sides and their
own, see e.g. I.4.3 in \cite{LV}. This simple fact gives a series of
useful consequences.

\begin{corollary}\label{corI} {\it Let $S$ be the open sector of the ring $A=\{ z\in\mathbb C:
r_1<|z-z_0|<r_2\}$, $z_0\in\mathbb C$, between the rays $R_k=\{
z\in\mathbb C: z=z_0+re^{i\alpha_k},\ r\in(0,\infty)\}$, $k=1,2$,
$0\le\alpha_1<\alpha_2\le2\pi$. Then
\begin{equation}\label{eqMODULI}
M(\Delta(R_1,R_2,S))\ =\
\frac{\log\frac{r_2}{r_1}}{\alpha_2-\alpha_1},\ \ \ \
M(\Delta(C_1,C_2,S))\ =\
\frac{\alpha_2-\alpha_1}{\log\frac{r_2}{r_1}}
\end{equation}
where $C_k$ are the boundary circles $\{ z\in\mathbb
C:|z-z_0|=r_k\}$, $k=1,2$, of the ring $A$.} \end{corollary}

\medskip

Indeed, the conclusion follows from the invariance of the modulus
$M$ under conformal mappings because the sector $S$ is mapped  by
$\log\ (z-z_0)$ onto the rectangle $R=\{ \zeta=\xi+i\eta\in\mathbb
C:\ \log r_1<\xi<\log r_2,\ \alpha_1<\eta<\alpha_2\}$.

\begin{corollary}\label{corII}{\it Under notations of Corollary \ref{corI} and $\alpha_2-\alpha_1=\Delta$,
the modulus of all Jordan arcs joining the rays $R_1$ and $R_2$ in
the sector $S$ is greater or equal to the number
$\frac{1}{\Delta}\log \frac{r_2}{r_1}$.}
\end{corollary}

\medskip

Indeed, every path $\gamma :[a,b]\to\mathbb C$ in
$\Delta(R_1,R_2,S)$ has a countable collection of loops because its
preimage (without the the corresponding point of cusp in $\mathbb
C$) is open in $(a,b)$. Thus, numbering its loops and removing them
by induction, we come to a Jordan arc $\gamma_*$ in
$\Delta(R_1,R_2,S)$ with its locus $|\gamma_*|\subseteq |\gamma|$.

\bigskip

\section{Some general topological lemmas}

The following statement is an analog of Proposition 2.3 in
\cite{RS}, see also Proposition 13.3 in \cite{MRSY}.

\begin{proposition}\label{pr10} Let $T$ be a topological space. Suppose that
$E_1$ and $E_2$ are sets in $T$ with
$\overline{E_1}\cap\overline{E_2}=\varnothing$. Then
\begin{equation}\label{eqC3}
\Delta\, (\, E_1\, ,\, E_2\, ,\, T\, )\ >\ \Delta\, (\, \partial
E_1\, ,\,
\partial E_2\, ,\, T\setminus(\overline E_1\cup\overline E_2)\, )\ .
\end{equation}
\end{proposition}

\begin{proof}
Indeed, let $\gamma\in\Delta ( E_1 , E_2 , T )$, i.e. the path
$\gamma: [a,b]\to T$ is such that $\gamma(a)\in E_1$ and
$\gamma(b)\in E_2$. Note that the set $\alpha :=
\gamma^{-1}(\overline{E_1})$ is a closed subset of the segment
$[a,b]$ because $\gamma$ is continuous, see e.g. Theorem 1 in
Section I.2.1 of \cite{Bou}. Consequently, $\alpha$ is compact
because $[a,b]$ is a compact space, see e.g. I.9.3 in \cite{Bou}.
Then there is $a_*:=\max\limits_{t\in \alpha}\, t<b$ because
 $\gamma(b)\in E_2$ and by the hypothesis of the proposition
$\overline{E_1}\cap\overline{E_2}=\varnothing$. Thus,
$\gamma^{\prime}:=\gamma |_{[a_*, b]}$ belongs to $\Delta (
\partial E_1 , E_2 , T\setminus \overline E_1 )$ because $\gamma$ is
continuous and hence $\gamma^{\prime}(a_*)$ cannot be an inner point
of $E_1$.

Arguing similarly in the space $T^{\prime}=T\setminus E_1$ with
$E_1^{\prime}:=E_2$ and $E_2^{\prime}:=\partial E_1$, we obtain that
there is $b_*:=\min\limits_{\gamma^{\prime}(t)\in \overline{E_2}} t>
a_*$. Thus, by the given construction $\gamma_* :=\gamma |_{[a_*,
b_*]}$ just belongs to $\Delta(\partial E_1,\partial
E_2,T\setminus(\overline E_1\cup\overline E_2)).$
\end{proof} $\Box$

\begin{lemma}\label{lem10}
In addition to the hypothesis of Proposition \ref{pr10}, let $T$ be
a subspace of a metric space $(M,\rho)$. Suppose that $$\partial
E_1\subseteq C_1:=\{ p\in M: \rho(p,p_0)=R_1\}, \ \ \ \partial
E_2\subseteq C_2:=\{ p\in M: \rho(p,p_0)=R_2\}$$ with $p_0\in
M\setminus T$ and $R_1<R_2$. Then
\begin{equation}\label{eqC}
\Delta\, (\, E_1\, ,\, E_2\, ,\, T\, )\ >\ \Delta\, (\, C_1\, ,\,
C_2\, ,\, A\, )
\end{equation}
where
$$
A\ =\ A(p_0,R_1,R_2)\ :=\ \{ p\in M:\ R_1 < \rho(p,p_0) < R_2\}\ .
$$
\end{lemma}

Note that here, generally speaking, $C_1\cap T\ne E_1$ and  $C_2\cap
T\ne E_2$ as well as $\gamma_*$ in the proof of Proposition
\ref{pr10} is not in $R$.

\bigskip

\begin{proof}
First of all, note that by the continuity of $\gamma_*$ the set
$\omega :=\gamma_*^{-1}(R)$ is open in $[a_*,b_*]$  and $\omega$ is
the union of a countable collection of disjoint intervals
$(a_1,b_1)$, $(a_2,b_2)$, $\ldots $ with ends in $\Gamma
:=\gamma^{-1}_*(\partial R)$. If there is a pair $a_k$ and $b_k$ in
the different sets $\Gamma_i:=\gamma^{-1}_*(C_i)$, $i=1,2$,
$\Gamma=\Gamma_1\cup\Gamma_2$, $\Gamma_1\cap\Gamma_2=\varnothing$,
then the proof is complete.

Let us assume that such a pair is absent. Then the given collection
is split into 2 collections of disjoint intervals $(a^{\prime}_l,
b^{\prime}_l)$ and $(a^{\prime\prime}_l, b^{\prime\prime}_l)$ with
ends $a^{\prime}_l, b^{\prime}_l\in\Gamma_1$ and
$a^{\prime\prime}_l, b^{\prime\prime}_l\in\Gamma_2$, $l=1,2,\ldots$.
Set $\alpha_1=\bigcup\limits_l(a^{\prime}_l, b^{\prime}_l)$ and
$\alpha_2=\bigcup\limits_l(a^{\prime\prime}_l, b^{\prime\prime}_l)$.

Arguing by contradiction, it is easy to show that
$\gamma_*:[a_*,b_*]\to (M,\rho)$ is uniformly continuous because
$[a_*,b_*]$ is a compact space. Indeed, let us assume that there is
$\varepsilon>0$ and a sequence of pairs $a^*_n$ and
$b^*_n\in[a_*,b_*]$, $n=1,2,\ldots$, such that $|b^*_n-a^*_n|\to 0$
as $n\to\infty$ and simultaneously
$\rho(\gamma_*(a^*_n),\gamma_*(b^*_n))\ge\varepsilon$. However, by
compactness of $[a_*,b_*]$ there is a subsequence $a^*_{n_k}\to
a_0\in[a_*,b_*]$ and then also  $b^*_{n_k}\to a_0$ as $k\to\infty$.
Hence by the continuity of $ \gamma_*$ it should be
$\rho(\gamma_*(a^*_{n_k}),\gamma_*(a_0))\to 0$ as well as
$\rho(\gamma_*(b^*_{n_k}),\gamma_*(a_0))\to 0$ and then by the
triangle inequality also
$\rho(\gamma_*(a^*_{n_k}),\gamma_*(b^*_{n_k}))\to 0$ as
$k\to\infty$. The contradiction disproves the assumption.

Note that $b_l^{\prime}-a_l^{\prime}\to 0$ as $l\to\infty$ and by
the uniform continuity of $\gamma_*$ on $[a_*,b_*]$ we have that
$|\gamma_l^{\prime}|\to C_1$ in the sense that $$\sup\limits_{p\in
|\gamma_l^{\prime}|}\ \inf\limits_{q\in C_1}\rho(p,q)\ \to\ 0\ \ \
 \ \hbox{as}\ \ l\to\infty$$ where
$\gamma_l^{\prime}:=\gamma_*|_{[a_l^{\prime}, b_l^{\prime}]}$,
$l=1,2,\ldots$. Thus, there is $R_2^{\prime}\in(R_1,R_2)$ such that
the set $L_1:=\bigcup\limits_l |\gamma_l^{\prime}|$ lies outside of
$B_2:=\{ p\in M: \rho(p,p_0)>R_2^{\prime}\}$.

Arguing similarly, we obtain that there is
$R_1^{\prime}\in(R_1,R^{\prime}_2)$ such that the set
$L_2:=\bigcup\limits_l |\gamma_l^{\prime\prime}|$ lies outside of
$B_1:=\{ p\in M: \rho(p,p_0)<R_1^{\prime}\}$. Remark that the sets
$\beta_1:=\gamma_*^{-1}(B_1)$ and $\beta_2:=\gamma_*^{-1}(B_2)$ are
open in $[a_*,b_*]$ because $\gamma_*$ is continuous and by the
construction $\delta_1:=\alpha_1\cup\beta_1$ and
$\delta_2:=\alpha_2\cup\beta_2$ are open, mutually disjoint and
together cover the segment $[a_*,b_*]$. The latter contradicts to
connectedness of the segment and, thus, disproves the above
assumption. \end{proof} $\Box$

\bigskip

\section{The main lemma}

\begin{lemma}\label{lem1} {\it Let $\mathbb{S}$ be a Riemann surface, $D$ be a domain in $\,
\overline{\mathbb{S}}$ with $\partial D\subset\mathbb{S}$ and let
$\Gamma$ be a isolated component of $\partial D$. Then $\Gamma$ has
a neighborhood $U$ with a conformal mapping $H$ of $U^*:=U\cap D$
onto a ring $R=\{ z\in\Bbb C: 0\le r<|z|<1\}$ where $\gamma
:=\partial U^*\cap D$ is a closed Jordan curve, $$C(\gamma , H)=\{
z\in\Bbb C:|z|=1\}\ ,\ \ \ C(\Gamma , H)=\{ z\in\Bbb C:|z|=r\}$$ and
$r=0$ if and only if $\Gamma$ is degenerated to a point.
Furthermore, the mapping $H$ can be extended to a homeomorphism
$\tilde H$ of $\ \overline{U^*}_P$ onto $\overline R$.} \end{lemma}

\medskip

Here we use the notation of the {\bf cluster set} of the mapping $H$
for $B\subseteq \partial D$,
$$C(B,H):=\left\{\ z\in{{\Bbb
C}}:z=\lim\limits_{k\to\infty}H(p_k),\ p_k\to p\in B,\ p_k\in D\
\right\}.$$

\medskip

{\bf Proof.} By the Kerekjarto--Stoilow representation of $\Bbb S$,
$\Gamma$ has an open neighborhood $V$ in $\Bbb S$ of a finite genus
and we may assume that $\overline V$ is a compact subset of $\Bbb
S$, $V$ is connected and does not intersect $\partial D\setminus
\Gamma$ because $\Gamma$ is an isolated component of $\partial D$.
Thus, $V\cap D$ is a Riemann surface of finite genus with an
isolated boundary element $g$ corresponding to $\Gamma$. However, a
Riemann surface of finite genus has boundary elements only  of the
first kind, see, e.g., IV.II.6 in \cite{St}. Consequently, $\Gamma$
has a neighborhood $U^*$ from the side of $D$ of genus zero with a
closed Jordan curve $\gamma =\partial U^*\cap D$. The latter means
that $U^*$ is homeomorphic to a plane domain and, consequently, by
the general principle of Koebe, see e.g. Section II.3 in \cite{KAG},
$U^*$ is conformally equivalent to a plane domain $D^*$.  Note that
by the construction $U^*$ has two nondegenerate boundary components.
Hence there is a conformal mapping $H$ of $U^*$ onto a ring
$D^*=R=\{ z\in\Bbb C: 0\le  r<|z|<1\}$ with $C(\gamma , H)=C_1:=\{
z\in\Bbb C:|z|=1\}$ and $C(\Gamma , H)=C_r:=\{ z\in\Bbb C:|z|=r\}$,
see e.g. Proposition 2.5 in \cite{RS} or Proposition 13.5 in
\cite{MRSY}. Set $U=U^*\cup (V\setminus D)$.

If $\Gamma$ is not degenerated into a point, then $r\ne 0$. Indeed,
in the contrary case the images of the closed Jordan curves around
the origin in the punctured disk $\mathbb D_{\varepsilon}=\{
z\in\mathbb C:0<|z|<\varepsilon\}$ under the mapping $H^{-1}$ should
be contracted to $\Gamma$ as $\varepsilon\to 0$ and hence their
lengths are not less than $\delta :=\mbox{diam}\,\Gamma >0$ for
small enough $\varepsilon$. However, the latter contradicts to the
conformal invariance of the modulus because by Corollary \ref{corII}
the modulus of all such closed Jordan curves is equal to $\infty$.
Inversely, if $\Gamma$ is degenerated into a point $p_0\in\mathbb
S$, then it is obvious that $r=0$ because $p_0$ has arbitrarily
small neighborhoods that are conformally mapped onto the unit disk
in $\mathbb C$. Hence we omit the consideration of this trivial case
and restrict ourselves by the case $r>0$.

Now, by the condition (i) in the definition of prime ends and the
invariance of $M$ we have, for every chain $\{ \sigma_m\}$ in a
prime end $P$ associated with $\Gamma$ and localized in $U^*$, that
\begin{equation}\label{eq4}
M(\Delta(\sigma_m,\sigma_{m+1}, U^*))\ <\ \infty\ \ \ \ \ \ \forall\
m=1,2,\ldots \end{equation} Moreover, by Remark \ref{rmk1} $P$
contains a chain $\{\sigma_m\}$ lying on circles $S_m=S(p_0,r_m)=\{
p\in\Bbb S :\ \delta(p,p_0)=r_m\}$ with $p_0\in\partial D$ and
$r_m\to0$ as $m\to\infty$ for which and any continuum $C$ in $U^*$
\begin{equation}\label{eq5}
\lim\limits_{m\to\infty}\ M(\Delta(\sigma_m,C, U^*))\ \leq
\lim\limits_{m\to\infty}\ M(\Delta(\sigma_m,\sigma_{m_0}, U^*))\ =\
0\ .
\end{equation}
Indeed, for every continuum $C$ in $U^*$, there is $m_0$ such that
$C\subset D\setminus d_{m_0}$ and the closed ball
$B_0=B(p_0,r_{m_0})=\{ p\in\Bbb S :\ \delta(p,p_0)\le r_{m_0}\}$ is
compact and lies in a chart $U_0$ of $p_0$. Then $\Delta(\sigma_m,C,
U^*)\subseteq \Delta(\sigma_m,D\setminus d_{m_0}, U^*)$, by
Proposition \ref{pr10} $\Delta(\sigma_m,D\setminus d_{m_0},
U^*)>\Delta(\sigma_m,\sigma_{m_0}, U^*)$ and by Lemma \ref{lem2}
$\Delta(\sigma_m,\sigma_{m_0}, U^*)>\Delta(S_m,S_{m_0}, A)$ where
$A:=\{ p\in\Bbb S :\ r_m<\delta(p,p_0)< r_{m_0}\}$ belongs to the
chart $U_0$ of the point $p_0$. Note, $M(\Delta(S_m,S_{m_0}, A))\le
M(\Delta(S_m,S_{m_0}, U_0)) \to 0$ as $m\to\infty$ because $S_{m_0}$
is a compact set in $B_0\setminus\{ p_0\}$ and $S_m$ is contracted
to $p_0$ as $m\to\infty$, see also 7.5 in \cite{Va}. Finally, we
obtain (\ref{eq5}) by the minorization principle, see e.g.
\cite{Fu}, p. 178. Similarly, it is proved that prime ends
associated with $\gamma$ also satisfy conditions (\ref{eq4}) and
(\ref{eq5}).

Thus, the prime ends of $U^*$ in the sense (i)--(iii) and their
images in $R$ are the prime ends in the sense of Section 4 in
\cite{Na}. By Lemma 3.5 in \cite{Na} the prime ends of N\"akki in
$R$ coincide with prime ends of Caratheodory. Moreover, the N\"akki
prime ends in $R$ has a one-to-one correspondence with the points of
$\partial R$ whose extension to the mapping between $\overline R$
and $\overline R_P$ by the identity in $R$ is a homeomorphism with
respect to the topologies of $\overline R$ and $\overline R_P$ or
with respect to convergence of points and prime ends, respectively,
see Theorems 4.1 and 4.2 in \cite{Na}. Consequently, if $p_k$ is a
sequence of points in $U^*$ which is convergent to a prime end $P$
of $U^*$, then $H(p_k)$ is convergent to a unique point
$z_0\in\partial R$ that depends only on $P$.

Denote by $\tilde H$ the extension of $H$ to $ \overline{U^*}_P$. It
is clear by definitions of prime ends of N\"akki and Caratheodory as
classes of equivalence that $\tilde H(P_1)\ne\tilde H(P_2)$ for
every prime ends $P_1\ne P_2$ of the domain $U^*$. Let us consider
the metric $\rho(P, P^*):=|\tilde H(P)-\tilde H(P^*)|$ on the space
$\overline {U^*}_P$. It is obvious by definitions that $\rho(P_k,
P_0)\to 0$ implies that $P_k\to P_0$ as $k\to\infty $. The inverse
conclusion follows because of the mapping $\tilde H:
\overline{U^*}_P\to\overline R$ is continuous. Indeed, let $P_k\to
P_0$, $k=1,2,\ldots $, be a sequence in $\overline{U^*}_P$. It is
obvious, $\tilde H(P_k)\to \tilde H(P_0)$ for $P_0\in U^*$. If
$P_0\in E_{U^*}$, then we are able to choose $p_k\in U^*$ such that
$|\tilde H(P_k)-\tilde H(p_k)|<2^{-k}$, $k=1,2,\ldots $, and $p_k\to
P_0$ as $k\to\infty$. The latter implies that $\tilde
H(p_k)\to\tilde H(P_0)$ and then the former implies that $\tilde
H(P_k)\to \tilde H(P_0)$. Thus, the space $\overline {U^*}_P$ is
metrizable with the given metric $\rho$ and $\tilde H$ is an
isometric embedding of $\overline {U^*}_P$ in $\overline R$. By
construction $\tilde H(U^*)=R$ and, by Proposition 2.5 in \cite{RS}
or Proposition 13.5 in \cite{MRSY}, $\tilde H(E_{U^*})\subseteq
\partial R$. Let us show that $\tilde H(E_{U^*})=\partial R$.

For this goal, fixing $z_0\in\partial C_r$ and
$\varepsilon\in(0,1)$, consider the family $\frak F$ of all Jordan
arcs in the open disk $B_{\varepsilon}=B(z_0,\varepsilon):=\{
z\in\mathbb C: |z-z_0|<\varepsilon\}$ joining in $R$ the two open
arcs $A_1$ and $A_2$ of $C_r\cap B_{\varepsilon}\setminus\{ z_0\}$.
By the minorization principle, see e.g. \cite{Fu}, and the
invariance of $M$ (with respect to the conformal mapping consisting
of the composition of the inversion with respect to the unit circle
and the reflection with respect to the straight line $L_0$ passing
through the origin and the point $z_0$) we obtain from Corollary
\ref{corII} that the conformal modulus of the family $\frak F$ is
equal to $\infty$. By the invariance of the modulus under conformal
mappings we have that the modulus of the family $\frak
F_*=H^{-1}(\frak F)$ is also equal to $\infty$. Consequently, the
length of elements of $\frak F_*$ cannot be restricted from below
and, by arbitrariness of $\varepsilon$, there is a sequence of
mutually disjoint cross-cuts $\sigma_m\in \frak F$ of $R$ with
$\sigma_m(0)\in A_1$ and $\sigma_m(1)\in A_2$ that is contracted to
the point $z_0$ such that $\delta(\sigma^*_m)\to 0$ as $m\to\infty$
where $\sigma^*_m = H^{-1}(\sigma_m)$ and, moreover,
$\sigma^*_{m+1}\subset d^*_m$ where $d^*_m$ is the corresponding
component of $D$ generated by $\sigma^*_m$, $\partial d^*_m\cap
U^*=\sigma^*_m$ for all $m=1,2,\ldots$. Note that such rectifiable
$\sigma^*_m:(0,1)\to D$ have limits $p^{(1)}_m=\lim\limits_{t\to
+0}\sigma^*_m(t)$ and $p^{(2)}_m=\lim\limits_{t\to
1-0}\sigma^*_m(t)$ because $\overline{U^*}$ is a compact subset of
${\mathbb S}$, see e.g. Proposition I.9.3 in \cite{Bou}, cf. also
Theorem 1.3.2 in \cite{Va}, moreover, the points $p^{(1)}_m$ and
$p^{(2)}_m$ belongs to $\Gamma$, see e.g. Proposition 2.5 in
\cite{RS} or Proposition 13.5 in \cite{MRSY}.

Finally, it remains to show that
$\overline{\sigma^*_m}\cap\overline{\sigma^*_{m+1}}=\varnothing$,
passing in case of need to a suitable subchain of cross--cuts
$\sigma_m$ in $R$. First of all, by the above construction we may
assume that
$$\delta_m\ :=\ \inf\limits_{z\in\sigma_m}|z-z_0|\ >\ \delta_m^*\
:=\ \sup\limits_{z\in\sigma_{m+1}}|z-z_0|\ >\ 0\ \ \ \ \ \ \ \
\forall m=1,2,\ldots$$ and also that $\sigma^*_m$ is contracted to a
point $p_0\in\Gamma$ because $\Gamma$ is compact and
$\delta(\sigma^*_m)\to 0$. It is clear that the desired subchain
exists if $\sigma^*_m(0)\ne p_0\ne\sigma^*_m(1)$ for all large
enough $m$.

In the contrary case, it would exist a subchain
$\tilde\sigma_k:=\sigma_{m_k}$, $k=1,2,\ldots$, such that either
$\tilde\sigma_k^*(0)=p_0=\tilde\sigma_{k+1}^*(0)$ or
$\tilde\sigma_k^*(1)=p_0=\tilde\sigma_{k+1}^*(1)$ for all
$k=1,2,\ldots$, where $\tilde\sigma_k^*:=H^{-1}(\tilde\sigma_k)$,
$k=1,2,\ldots$. In the first case, consider the ring $A=\{
z\in\mathbb C: r_1<|z-z_0|<r_2\}$ with
$0<\delta_{m_k}^*<r_1<r_2<\delta_{m_k}$. As above, by the
minorization principle, the invariance of $M$ and Corollary
\ref{corI} the conformal modulus of the family $\tilde {\frak F}$ of
all paths in $A\cap R$ joining the open arc $A_0:=A\cap A_1$ of the
circle $C_r$ and the interval ${\rm I}_0:=A\cap L_0$ of the straight
line $L_0$ is not less than $\frac{2}{\pi}\log\frac{r_2}{r_1}>0$.
The modulus of the family $\tilde {\frak F}_*=H^{-1}(\tilde {\frak
F})$ should be the same. However, the modulus of $\tilde {\frak
F}_*$ is equal to zero because all paths in $\tilde {\frak F}_*$ are
ended at the point $p_0$.

Indeed, denote by $\rm I$ the maximal open interval of $L_0$
containing $\rm I_0$ and not intersecting $\tilde\sigma_k$ and
$\tilde\sigma_{k+1}$, and by $t_0$ and $t_*$ the parameter numbers
in $(0,1)$ corresponding to its ends on $\tilde\sigma_k$ and
$\tilde\sigma_{k+1}$. Then $H^{-1}(\rm I)$,
$\tilde\sigma^*_k((0,t_0])$, $\tilde\sigma^*_{k+1}((0,t_*])$ and the
point $p_0$ form a closed Jordan curve in $\overline {U^*}$ with the
only point on $\partial {U^*}$. Note that the corresponding Jordan
domain contains the family $\tilde {\frak F}_*$ of paths $\gamma$
that should be ended on $\Gamma$ and, consequently, at the point
$p_0$. The second possibility is similarly disproved.

\bigskip

Thus, $\tilde H$ is isometry between $\overline {U^*}_P$ with the
given metric $\rho$ and $\overline R$. $\Box$

\bigskip

\begin{remark}\label{rmk2} {\rm By the proof we have that $\overline {U^*}_P$ is a compact
space with the metric $\rho$. Moreover, it follows from the proof
that the spaces of prime ends by Caratheodory and N\"akki coincide
not only in the ring $R$ but also in $U^*$ because the N\"akki prime
ends are invariant under conformal mappings.

Furthermore, if $D$ be a domain in the Kerekjarto-Stoilow
compactification $\overline{\Bbb S}$ of a Riemann surface $\Bbb S$
and $\partial D$ is a set in $\Bbb S$ with a finite collection of
components, then their prime ends by Caratheodory and N\"akki also
coincide, the whole space $\overline D_P$ can be metrized through
the theory of pseudometric spaces, see e.g. Section 2.21.XV in
\cite{Ku}, and  $\overline D_P$ is compact.

Namely, let $\rho_0$ be one of the metrics on $\overline{\Bbb S}$
and let $\rho_1,\ldots , \rho_n$ be the above metrics on
$\overline{U_1^*}_P, \ldots , \overline{U_n^*}_P$ for the
corresponding components $\Gamma_1, \ldots , \Gamma_n$ of $\partial
D$. Here we may assume that the sets $\overline {U^*_j}$ are
mutually disjoint. Then $\rho^*_j:=\rho_j/(1+\rho_j)\leq 1$, $j=0,
1,\ldots , n$, are also metrics generating the same topologies on
$D_0:=D\setminus \cup U^*_j$, $\overline{U_1^*}_P, \ldots ,
\overline{U_n^*}_P$, correspondingly, see e.g. Section 2.21.V in
\cite{Ku}, and the topology of prime ends on $\overline D_P$ is
generated by the metric $\rho = \sum\limits_{j=0}\limits^n
2^{-(j+1)}\tilde\rho_j< 1$ where the pseudometrics $\tilde\rho_j$
are extensions of $\rho_j^*$ onto $\overline D_P$ by $1$, see e.g.
Remark 2 in point 2.21.XV of \cite{Ku}. Note that the space
$\overline D_P$ is compact because $\overline D_P=\cup\
\overline{U^*_j}_P\cup D_0$ where $D_0$ is a compact space as a
closed subset of the compact space $\overline{\Bbb S}$, see e.g.
Proposition I.9.3 in \cite{Bou}. } \end{remark}

\bigskip

\begin{corollary}\label{corIII}
{\it Under hypothesis of Lemma \ref{lem1}, the space of all prime
ends associated with a nondegenerate isolated component of $\partial
D$ is homeomorphic to a circle. } \end{corollary}

\bigskip

\section{On boundary behavior in prime ends of inverse maps}

The main base for extending inverse mappings is the following fact.

\begin{lemma}\label{lem3} {\it Let $\mathbb{S}$ and $\ \mathbb{S}^{\prime}$ be Riemann surfaces, $D$ and
$D^{\prime}$ be domains in $\, \overline{\mathbb{S}}$ and $\,
\overline{\mathbb{S}^{\prime}}$, $\partial D\subset\mathbb{S}$ and
$\partial D^{\prime}\subset\mathbb{S}^{\prime}$ have finite
collections of components, and let $f:D\to D^*$ be a homeomorphism
of finite distortion with $\ K_{f}\in L^{1}_{\rm loc}$. Then
\begin{equation}\label{eqC} C(P_1,f)\cap C(P_2,f)=\varnothing
\end{equation} for all prime ends $P_1\ne P_2$ in the domain $D$.}\end{lemma}

Here we use the notation of the {\bf cluster set} of the mapping $f$
at $P\in E_D$,
$$C(P,f):=\left\{\ P^{\prime}\in E_{D^{\prime}}: P^{\prime}=\lim\limits_{k\to\infty}f(p_k),\ p_k\to P,\
p_k\in D\ \right\}$$

As usual, we also assume here that the dilatation $K_{f}$
of the mapping $f$ is extended by zero outside of the domain $D$.

\medskip

\begin{proof} First of all note that
$\overline{\mathbb{S}}$ and $\overline{\mathbb{S}^{\prime}}$ are
metrizable spaces. Hence their compactness is equi\-va\-lent to
their sequential compactness, see e.g. Remark 41.I.3 in
\cite{Ku$_2$}, and, consequently, $\partial{D}$ and
$\partial{D}^{\prime}$ are compact subsets of ${\mathbb{S}}$ and
${\mathbb{S}^{\prime}}$, correspondingly, see e.g. Proposition I.9.3
in \cite{Bou}. Thus, by Lemma \ref{lem1},  Remarks \ref{rmk1} and
\ref{rmk2} we may assume that $\overline D$ is a compact set in
$\Bbb S$, $K_f\in L^1(D)$, $P_1$ and $P_2$ are associated with the
same component $\Gamma$ of $\partial D$ and $D^{\prime}$ is a ring
$R=\{ z\in\Bbb C: 0<r<|z|<1\}$ and
$$
A_k\ :=\ C(P_k, f)\ ,\ \ \ \ \ \ k\ =\ 1,\ 2
$$
are sets of points in the circle $C_r:=\{ z\in\Bbb C:|z|=r\}$,
$\partial D$ consists of 2 components: $\Gamma$ and a closed Jordan
curve $\gamma$, $f$ is extended to a homeomorphism of $D\cup\gamma$
onto $D^{\prime}\cup C_1$, $C(C_r,f^{-1})=\Gamma$, see also
Proposition 2.5 in \cite{RS} or Proposition 13.5 in \cite{MRSY}.
Note that the sets $A_k$ are continua, i.e. closed arcs of the
circle $C_r$, because
$$
A_k\ =\
\bigcap\limits_{m=1}\limits^{\infty}\overline{f\left(d^{(k)}_m\right)}\
,\ \ \ \ \ \ k\ =\ 1,\ 2\ ,
$$
where $d^{(k)}_m$ are domains corresponding to chains of cross--cuts
$\{\sigma^{(k)}_m\}$ in the prime ends $P_k$, $k=1,2$, see e.g.
I(9.12) in \cite{Wh} and also I.9.3 in \cite{Bou}. In addition, by
Remark \ref{rmk1} we may assume also that $\sigma^{(k)}_m$ are open
arcs of the circles $C^{(k)}_{m}:=\{ p\in\Bbb S:
h(p,p_k)=r^{(k)}_m\}$ on $\Bbb S$ with $p_k\in\partial D$ and
$r^{(k)}_m\to 0$ as $m\to\infty$, $k=1,2$.

Set $p_0=p_1$. By the definition of the topology of the prime ends
in the space $\overline D_P$, we have that  $d^{(1 )}_m\cap d^{(2
)}_m=\varnothing$ for all large enough $m$ because $P_1\ne P_2$. For
a such $m$, set $R_1=r^{(1)}_{m+1}<R_2=r^{(1)}_m$ and
$$
U_k=d^{(k)}_m\ , \ \ \  \Sigma_k=\sigma^{(k)}_m\ ,\ \ \ C_k=\{
p\in\Bbb S: h(p,p_0)=R_k\},\ \ \ k=1,2\ .
$$
Let $K_1$ and $K_2$ be arbitrary continua in $U_1$ and $U_2$,
correspondingly. Applying Proposition \ref{pr10} and Lemma
\ref{lem10} with $T=D$, $E_1=d^{(1)}_{m+1}$ and $E_2=D\setminus
d^{(1)}_m$, and taking into account the inclusion $\Delta(K_1,K_2,
D)\subset \Delta(E_1,E_2, D)$, we obtain that
\begin{equation}\label{eqM}\Delta(K_1,K_2, D)\
>\ \Delta(C_1,C_2, A)\ ,\ \ \  A:=\{ p\in\mathbb{S}:
R_1<h(p,p_0)<R_2 \}\ ,\end{equation} which means that any path
$\alpha :[a,b]\to\Bbb S$ joining $K_1$ and $K_2$ in $D$,
$\alpha(a)\in K_1$, $\alpha(b)\in K_2$ and $\alpha(t)\in D$,
$t\in(a,b)$, has a subpath joining $C_1$ and $C_2$ in $A$. Thus,
since $f$ is a homeomorphism, we have also that
\begin{equation}\label{eqFV}\Delta(fK_1,fK_2, fD)\
>\ \Delta(fC_1,fC_2, fA)\end{equation}
and by the minorization principle, see e.g. \cite{Fu}, p. 178, we
obtain that
\begin{equation}\label{eqMFV}M(\Delta(fK_1,fK_2, fD))\
\leq\ M(\Delta(fC_1,fC_2, fA))\ .\end{equation} So, by Lemma 3.1 in
\cite{RV3} we conclude that
\begin{equation}\label{eqESTIMATE} M(\Delta(fK_1,fK_2, fD))\
\leqslant\ \int\limits_{A}K_f(p)\cdot\ \xi^2(h(p,p_0))\
dh(p)\end{equation} for all measurable functions
$\xi:(R_1,R_2)\to[0,\infty]$ such that
\begin{equation}\label{eqOS1.9}\int\limits_{R_1}^{R_2}\xi(R)\ dR\geqslant\ 1\ .\end{equation}
In particular, for $\xi(R)\equiv 1/\delta$, $\delta = R_2-R_1>0$, we
get from here that
\begin{equation}\label{eqFINITE} M(\Delta(fK_1,fK_2, fD))\
\leqslant\ M_0\ :=\ \frac{1}{\delta}\int\limits_{D}K_f(p)\ dh(p)\ <\
\infty \ .\end{equation} Since $f$ is a homeomorphism,
(\ref{eqFINITE}) means that \begin{equation}\label{eqFINAL}
M(\Delta({\cal K}_1,{\cal K}_2, D^{\prime}))\ \leqslant\ M_0\ <\
\infty \end{equation} for all continua ${\cal K}_1$ and ${\cal K}_2$
in the domains $V_1=fU_1$ and $V_2=fU_2$, correspondingly.

Let us assume that $A_1\cap A_2\ne\varnothing$. Then by the
construction there is $p_0\in\partial R\cap\partial V_1\cap\partial
V_2$. However, the latter contradicts (\ref{eqFINAL}) because the
ring $R$ is a QED (quasiextremal distance) domains, see e.g. Theorem
3.2 in \cite{MRSY}, see also Theorem 10.12 in \cite{Va}. \end{proof}
$\Box$

\begin{theorem}\label{th1}
Let $\mathbb{S}$ and $\ \mathbb{S}^{\prime}$ be Riemann surfaces,
$D$ and $D^{\prime}$ be domains in $\, \overline{\mathbb{S}}$ and
$\, \overline{\mathbb{S}^{\prime}}$, correspondingly, $\partial
D\subset\mathbb{S}$ and $\ \partial
D^{\prime}\subset\mathbb{S}^{\prime}$ have finite collections of
nondegenerate components, and let $f:D\to D^{\prime}$ be a
homeomorphism of finite distortion with $\ K_{f}\in L^{1}_{\rm
loc}$. Then the inverse mapping $g=f^{-1}:D^{\prime}\to D$ can be
extended to a continuous mapping $\tilde g$ of
$\overline{D^{\prime}}_P$ onto $\overline{D}_P$.
\end{theorem}


\begin{proof} Recall that by Remark \ref{rmk2} the spaces $\overline D_P$ and
$\overline{D^{\prime}}_P$ are compact and metrizable with metrics
$\rho$ and $\rho^{\prime}$. Let a sequence $p_n\in D^{\prime}$
converges  as $n\to\infty$ to a prime end $P^{\prime}\in
E_{D^{\prime}}$. Then any subsequence of $p^*_n:=g(p_n)$ has a
convergent subsequence by compactness of $\overline D_P$. By Lemma
\ref{lem3} any such convergent subsequence should have the same
limit. Thus, the sequence $p^*_n$ is convergent, see e.g. Theorem 2
of Section 2.20.II in \cite{Ku}. Note that $p^*_n$ cannot converge
to an inner point of $D$ because $I(P)\subseteq\partial D$ by
Proposition \ref{pr1} and, consequently, $p_n$ is convergent to
$\partial D^{\prime}$, see e.g. Proposition 2.5 in \cite{RS} or
Proposition 13.5 in \cite{MRSY}. Thus, $E_{D^{\prime}}$ is mapped
into $E_D$ under this extension $\tilde g$ of $g$. In fact, $\tilde
g$ maps $E_{D^{\prime}}$ onto $E_D$ because $p_n=f(p_n^*)$ has a
convergent subsequence for every sequence $p_n^*\in D$ that is
convergent to a prime end $P$ of the domain $D$ because
$\overline{D^{\prime}}_P$ is compact. The map $\tilde g$ is
continuous. Indeed, let a sequence $P_n^{\prime}\in
\overline{D^{\prime}}_P$ be convergent to $P^{\prime}\in
\overline{D^{\prime}}_P$. Then there is a sequence $p_n\in
D^{\prime}$ such that $\rho^{\prime}(P^{\prime}_n, p_n)<2^{-n}$ and
$\rho(p^*_n, P_n^*)<2^{-n}$ where $p^*_n:=g(p_n)$, $P^*_n:=\tilde
g(P_n)$ and $P^*=\tilde g(P^{\prime})$. Then $p_n\to P^{\prime}$ and
by the above $p^{*}_n\to P^*$ as well as $P^{*}_n\to P^*$ as
$n\to\infty$.\end{proof} $\Box$

\bigskip

\section{Lemma on extension to boundary of direct mappings}

In contrast with the case of the inverse mappings, as it was already
established in the plane, no degree of integrability of the
dilatation leads to the extension to the boundary of direct mappings
with finite distortion, see the example in the proof of Proposition
6.3 in \cite{MRSY}. The nature of the corresponding conditions has a
much more refined character as the following lemma demonstrates.

\begin{lemma}\label{lem4} {\it Under the hypothesis of Theorem \ref{th1}, let
in addition
\begin{equation}\label{eqREFINED}
\int\limits_{R(p_0,\varepsilon, \varepsilon_0)}
K_f(p)\cdot\psi_{p_0,\varepsilon, \varepsilon_0}^2(h(p,p_0))\ dh(p)
= o\left(I_{p_0, \varepsilon_0}^2(\varepsilon)\right)\ \ \ \ \
\forall\ p_0\in\partial D
\end{equation} as $\varepsilon\to0$ for all
$\varepsilon_0<\delta(p_0)$ where $R(p_0,\varepsilon,
\varepsilon_0)=\{p\in \Bbb S:\varepsilon<h(p,p_0)<\varepsilon_0\}$
and $\psi_{p_0,\varepsilon, \varepsilon_0}(t): (0,\infty)\to
[0,\infty]$, $\varepsilon\in(0,\varepsilon_0)$, is a family of
measurable functions such that
$$0\ <\ I_{p_0, \varepsilon_0}(\varepsilon)\ :=\
\int\limits_{\varepsilon}^{\varepsilon_0}\psi_{p_0,\varepsilon,
\varepsilon_0}(t)\ dt\ <\ \infty\qquad\qquad\forall\
\varepsilon\in(0,\varepsilon_0)\ .$$ Then $f$ can be extended to a
continuous mapping $\tilde f$ of $\ \overline{D}_P$ onto $\
\overline{D^{\prime}}_P$.}
\end{lemma}

We assume here that the function $K_f$ is extended by zero outside
of $D$.

\medskip

\begin{proof} By and Lemma \ref{lem1}, Remarks \ref{rmk1} and \ref{rmk2}, arguing as in the
beginning of the proof of Lemma \ref{lem3}, we may assume that
$\overline D$ is a compact set in $\Bbb S$, $\partial D$ consists of
2 components: a closed Jordan curve $\gamma$ and one more
non\-de\-ge\-ne\-ra\-te component $\Gamma$, $D^{\prime}$ is a ring
$R=\{ z\in\Bbb C: 0<r<|z|<1\}$,
$\overline{D^{\prime}}_P=\overline{R}$,
$$C(\Gamma , f)=C_r:=\{ z\in\Bbb C: |z|=r\},\ \ \ C(\gamma , f)=C_*:=\{
z\in\Bbb C: |z|=1\}$$ and that $f$ is extended to a homeomorphism of
$D\cup\gamma$ onto $D^{\prime}\cup C_*$.

Let us first prove that the set $L:=C(P,f)$ consists of a single
point of $C_r$ for a prime end $P$ of the domain $D$ associated with
$\Gamma$. Note that $L\neq\varnothing$ by compactness of the set
$\overline{R}$ and, moreover, $L\subseteq C_r$ by Proposition
\ref{pr1}.


Let us assume that there is at least two points $\zeta_0$ and
$\zeta_*\in L$. Set $U=\{ \zeta\in\Bbb C: |\zeta -\zeta_0|<\rho_0\}$
where $0<\rho_0<|\zeta_*-\zeta_0|$.

\medskip

Let $\sigma_k$, $k=1,2,\ldots\,$, be a chain in the prime end $P$
from Remark \ref{rmk1} lying on the circles $S_k:=\{ p\in\Bbb S:
h(p,p_0)=r_k\}$ where $p_0 \in\Gamma$ and $r_k\to 0$ as
$k\to\infty$. Let $d_k$ be the domains associated with $\sigma_k$.
Then there exist points $\zeta_k$ and $\zeta^*_k$ in the domains
$d_{k}^{\prime}=f(d_{k})\subset R$ such that
$|\zeta_0-\zeta_k|<\rho_0$ and $|\zeta_0-\zeta^*_k|>\rho_0$ and,
moreover, $\zeta_k\to \zeta_0$ and $\zeta^*_k\to \zeta_*$ as
$k\to\infty$. Let $\gamma_k$ be paths joining $\zeta_k$ and
$\zeta^*_k$ in $d_{k}^{\prime}$. Note that by the construction
$\partial U\cap \gamma_k\neq\varnothing$, $k=1,2,\ldots $.

By the condition of strong accessibility of the point $\zeta_0$ in
the ring $R$, there is a continuum $E\subset R$ and a number
$\delta>0$ such that
\begin{equation}\label{delta} M(\Delta(E,\gamma_k;R))\ \geqslant\ \delta\end{equation}
for all large enough $k$. Note that $C=f^{-1}(E)$ is a compact
subset of $D$ and hence $h(p_0,C))>0$. Let
$\varepsilon_0\in(0,\delta_0)$ where $\delta_0:=\min\, (\delta(p_0),
h(p_0,C))$. Without loss of generality, we may assume that
$r_k<\varepsilon_0$ and that (\ref{delta}) holds for all
$k=1,2,\ldots$.

Let $\Gamma_{m}$ be the family of paths joining the circle
$S_{0}:=\{ p\in\Bbb S: h(p,p_0)=\varepsilon_0\}$ and ${\sigma_m}$,
$m=1,2,\ldots$,  in the intersection of $D\setminus d_m$ and the
ring $R_{m}:=\{ p\in\Bbb S: r_m < h(p,p_0) < \varepsilon_0\}$.
Applying Proposition \ref{pr10} and Lemma \ref{lem10} with $T=D$,
$E_1=d_m$ and $E_2=$ $B_0:=\{ p\in\Bbb S: h(p,p_0) >
\varepsilon_0\}$, and taking into account the inclusion
$\Delta(C,C_k,D)\subset \Delta(E_1,E_2,D)=\Delta(B_0,d_m,D)$ where
$C_k=f^{-1}(\gamma_k)$, we have that $\Delta(C,C_k,D)>\Gamma_m$ for
all $k\geqslant m$ because by the construction $C_k\subset
d_k\subset d_m$. Thus, since $f$ is a homeomorphism, we have also
that $\Delta(E,\gamma_k,D)>f\Gamma_m$ for all $k\geqslant m$, and by
the principle of minorization, see e.g. \cite{Fu}, p. 178, we obtain
that $M(f(\Gamma_{m}))\geqslant\delta$ for all $m=1,2,\ldots$.

On the other hand, every function $\xi(t)=\xi_m(t):=
\psi_{p_0,r_m,\varepsilon_0}(t)/I_{p_0,\varepsilon_0}(r_m)$,
$m=1,2,\ldots$, satisfies the condition (\ref{eqOS1.9}) and by Lemma
3.1 in \cite{RV3}
$$M(f\Gamma_m)\ \leqslant\
\int\limits_{R_m} K_f(p)\cdot\xi_{m}^2(h(p,p_0))\ dh(p)\ ,
$$
i.e., $M(f\Gamma_m)\to 0$ as $m\to\infty$ in view of
(\ref{eqREFINED}).


The obtained contradiction disproves the assumption that the cluster
set $C(P,f)$ consists of more than one point.


Thus, we have the extension $\tilde f$ of $f$ to $\overline{D}_P$
such that $\tilde f(E_D)\subseteq E_{D^{\prime}}$. In fact, $\tilde
f(E_D)= E_{D^{\prime}}$. Indeed, if $\zeta_0\in D^{\prime}$, then
there is a sequence $\zeta_n$ in $D^{\prime}$ that is convergent to
$\zeta_0$. We may assume with no loss of generality that
$f^{-1}(\zeta_n)\to P_0\in \overline{D}_P$ because $\overline{D}_P$
is compact, see Remark \ref{rmk2}. Hence $\zeta_0\in E_D$ because
$\zeta_0\notin D$, see e.g. Proposition 2.5 in \cite{RS} or
Proposition 13.5 in \cite{MRSY}.

Finally, let us show that the extended mapping $\tilde
f:\overline{D}_P\to\overline{D^{\prime}}_P$ is continuous. Indeed,
let $P_n\to P_0$ in $\overline{D}_P$. The statement is obvious for
$P_0\in D$. If $P_0\in E_D$, then by the last item we are able to
choose $P^*_n\in D$ such that $\rho(P_n,P^*_n)<2^{-n}$ and
$\rho^{\prime}(\tilde f(P_n),\tilde f(P^*_n))<2^{-n}$ where $\rho$
and $\rho^{\prime}$ are some metrics on $\overline{D}_P$ and
$\overline{D^{\prime}}_P$, correspondingly, see Remark \ref{rmk2}.
Note that by the first part of the proof $f(P^*_n)\to f(P_0)$
because $P^*_n\to P_0$. Consequently, $\tilde f(P_n)\to \tilde
f(P_0)$, too. \end{proof} $\Box$

\begin{remark}\label{rmk3} {\rm Note that condition (\ref{eqREFINED}) holds, in
particular, if
\begin{equation}\label{eqOSKRSS100a} \int\limits_{D(p_0,
\varepsilon_0)}K_f(p)\cdot\psi^2 (h(p,p_0))\ dh(p)\ <\
\infty\qquad\qquad \forall\ p_0\in\partial D\end{equation} where
$D(p_0,\varepsilon_0)=\{p\in \Bbb S: h(p,p_0)<\varepsilon_0\}$ and
where $\psi(t): (0,\infty)\to [0,\infty]$ is a locally integrable
function such that $I_{p_0,\varepsilon_0}(\varepsilon)\to\infty$ as
$\varepsilon\to0$. In other words, for the extendability of $f$ to a
continuous mapping of $\overline{D}_P$ onto
$\overline{D^{\prime}}_P$, it suffices for the integrals in
(\ref{eqOSKRSS100a}) to be convergent for some nonnegative function
$\psi(t)$ that is locally integrable on $(0,\infty)$ but that has a
non-integrable singularity at zero.}
\end{remark}


\section{On the homeomorphic extension to the boundary}


Combining Lemma \ref{lem4} and Theorem \ref{th1}, we obtain the
significant conclusion:

\begin{lemma}\label{lem5} {\it $\ $  Under $\ $ the\ hypothesis$\ $ of\
Lemma\ \ref{lem4},$\ $ the\ homeomorphism  $f:D\to D^{\prime}$ can
be extended to a homeomorphism $\tilde f:
\overline{D}_P\to\overline{D^{\prime}}_P$.}\end{lemma}

\begin{proof} Indeed, by Lemma \ref{lem3} the mapping $\tilde f:
\overline{D}_P\ \to \overline{D^{\prime}}_P$ from Lemma \ref{lem4}
is injective and hence it has the well defined inverse mapping
$\tilde f^{-1}:\overline{D^{\prime}}_P\ \to \overline{D}_P$ and the
latter coincides with the mapping $\tilde g:\overline{D^{\prime}}_P\
\to \overline{D}_P$ from Theorem \ref{th1} because a limit under a
metric convergence is unique. The continuity of the mappings $\tilde
g$ and $\tilde f$ follows from Theorem \ref{th1} and Lemma
\ref{lem4}, respectively. \end{proof} $\Box$

\bigskip

We assume everywhere in this section that the function $K_f$ is
extended by zero outside of $D$.


\begin{theorem}\label{th2} {\it Under the hypothesis of
Theorem \ref{th1}, let in addition
\begin{equation}\label{e:6.4dc}\int\limits_{0}^{\varepsilon_0}
\frac{dr}{||K_{f}||(p_0,r)}\ =\ \infty\qquad\qquad \forall\
p_0\in\partial D,\ \ \ \ \varepsilon_0<\delta(p_0) \end{equation}
where
\begin{equation}\label{e:6.6ch} ||K_{f}||(p_0,r)\ :=\int\limits_{
S({p_0},r)}K_f(p)\ ds_h(p)\ .\end{equation} Then $f$ can be extended
to a homeomorphism of $\ \overline{D}_P$ onto
$\overline{D^{\prime}}_P$.}
\end{theorem}

Here $S(p_0,r)$ denotes the circle $\{p\in \Bbb S: h(p,p_0)=r\}$.

\bigskip

\begin{proof} Indeed, for the functions
\begin{equation}\label{3.21}\psi_{p_0, \varepsilon_0}(t)\ :=\ \left \{\begin{array}{lr}
1/||K_{f}||(p_0,t),\quad & \ t\in (0,\varepsilon_0),
\\ 0,  & \ t\in [\varepsilon_0,\infty),\end{array}\right.\end{equation} we
have by the Fubini theorem that
\begin{equation}\label{3.22}\int\limits_{R(p_0,\varepsilon, \varepsilon_0)} K_{f}(p)\cdot\psi_{p_0,\varepsilon_0}^2(h(p,p_0)\
dh(p)\ =\ \int\limits_{\varepsilon}^{\varepsilon_0}
\frac{dr}{||K_{f}||(p_0,r)}
\end{equation} where $R(p_0,\varepsilon, \varepsilon_0)$ denotes the ring $\{p\in \Bbb S:\varepsilon<h(p,p_0)<\varepsilon_0\}$ and,
consequently, condition (\ref{eqREFINED}) holds by (\ref{e:6.4dc})
for all $p_0\in\partial D$ and
$\varepsilon_0\in(0,\varepsilon(p_0))$.

\medskip

Here we have used the standard conventions in the integral theory
that $a/\infty=0$ for $a\neq\infty$ and $0\cdot\infty=0$, see, e.g.,
Section I.3 in \cite{Sa}.

\bigskip

Thus, Theorem \ref{th2} follows immediately from Lemma \ref{lem4}.
\end{proof} $\Box$


\begin{corollary}\label{cor1} {\it In particular, the conclusion of Theorem
\ref{th2} holds if
\begin{equation}\label{eqOSKRSS100d}k_{p_0}(r)=O\left(\log{\frac1r}\right)\qquad\qquad  \forall\
p_0\in\partial D \end{equation} as $r\to0$ where $k_{p_0}(r)$ is the
average of $K_f$ over the infinitesimal circle $S(p_0,r)$. }
\end{corollary}

Choosing in (\ref{eqREFINED}) $\psi(t):=\frac{1}{t\log 1/t}$, we
obtain by Lemma \ref{lem4} the next result, see also Lemma 4.1 in
\cite{RS} or Lemma 13.2 in \cite{MRSY}.


\begin{theorem}\label{th3} {\it Under the hypothesis of
Theorem \ref{th1}, let $K_{f}$ have a dominant $Q_{p_0}$ in a
neighborhood of each point $p_0\in \partial D$ with finite mean
oscillation at $p_0$. Then $f$ can be extended to a homeomorphism
$\tilde f:\overline{D}_P\to\overline{D'}_P$.}
\end{theorem}

By Corollary 4.1 in \cite{RS} or Corollary 13.3 in \cite{MRSY} we
obtain the following.


\begin{corollary}\label{cor2} {\it In particular, the conclusion of Theorem \ref{th3} holds if
\begin{equation}\label{eqOSKRSS6.6.3}
\overline{\lim\limits_{\varepsilon\to0}}\ \
\dashint_{D(p_0,\varepsilon)}K_{f}(p)\ dh(p)\ <\ \infty\qquad\qquad
\forall\ p_0\in\partial D\end{equation} where $D(p_0,\varepsilon)$
is the infinitesimal disk $\ \{p\in \Bbb S:\ h(p,p_0)\ <\
\varepsilon\}$.}
\end{corollary}


\begin{corollary}\label{cor3} {\it The conslusion of Theorem \ref{th3} holds if
every point $p_0\in\partial D$ is a Lebesgue point of the function
$K_{f}$ or its dominant $Q_{p_0}$.}
\end{corollary}


The next statement also follows from Lemma \ref{lem4} under the
choice $\psi(t)=1/t.$


\begin{theorem}\label{th4} {\it Under the hypothesis of
Theorem \ref{th1}, let, for some $\varepsilon_0>0$,
\begin{equation}\label{eqOSKRSS10.336a}\int\limits_{\varepsilon<h(p,p_0)<\varepsilon_0}K_{f}(p)\
\frac{dh(p)}{h^2(p,p_0)}\ =\
o\left(\left[\log\frac{1}{\varepsilon}\right]^2\right)\ \ \
\mbox{as}\ \ \varepsilon\to 0\ \ \ \forall\ p_0\in\partial
D\end{equation} Then $f$ can be extended to a homeomorphism of $\
\overline{D}_P$ onto $\overline{D'}_P$.}
\end{theorem}


\begin{remark}\label{rmOSKRSS200} {\rm Choosing in Lemma \ref{lem4}
the function $\psi(t)=1/(t\log 1/t)$ instead of $\psi(t)=1/t$,
(\ref{eqOSKRSS10.336a}) can be replaced by the more weak condition
\begin{equation}\label{eqOSKRSS10.336b}
\int\limits_{\varepsilon<h(p,p_0)<\varepsilon_0}\frac{K_{f}(p)\
dh(p)}{\left(h(p,p_0)\ \log{\frac{1}{h(p,p_0)}}\right) ^2}\ =\
o\left(\left[\log\log\frac{1}{\varepsilon}\right]^2\right)\end{equation}
and (\ref{eqOSKRSS100d}) by the condition
\begin{equation}\label{eqOSKRSS10.336h} k_{p_0}(r)\ =\ o
\left(\log\frac{1}{\varepsilon}\log\,\log\frac{1}{\varepsilon}
\right).\end{equation} Of course, we could give here the whole scale
of the corresponding condition of the logarithmic type using
suitable functions $\psi(t).$}
\end{remark}


\section{On interconnections between integral conditions}

For every non-decreasing function $\Phi:[0,\infty ]\rightarrow
[0,\infty ]$, the {\bf inverse function} $\Phi^{-1}$ can be well
defined by setting \begin{equation}\label{eq5.5CC} \Phi^{-1}(\tau)\
=\ \inf\limits_{\Phi(t)\ge \tau}\ t\ . \end{equation} As usual, here
$\inf$ is equal to $\infty$ if the set of $t\in[0,\infty ]$ such
that  $\Phi(t)\ge \tau$ is empty. Note that the function $\Phi^{-1}$
is non-decreasing, too.

\begin{remark}\label{rmk3.333} {\rm Immediately by the definition it
is evident that \begin{equation}\label{eq5.5CCC} \Phi^{-1}(\Phi(t))\
\le\ t\ \ \ \ \ \ \ \ \forall\ t\in[ 0,\infty ]
\end{equation} with the equality in (\ref{eq5.5CCC}) except intervals
of constancy of the function $\Phi(t)$.}\end{remark}

Recall that a function $\Phi :[0,\infty ]\rightarrow [0,\infty ]$ is
called convex if
$$
\Phi (\lambda t_1 + (1-\lambda) t_2)\ \le\ \lambda\ \Phi (t_1)\ +\
(1-\lambda)\ \Phi (t_2)
$$
for all $t_1,$ $t_2\in[0,\infty ]$ and $\lambda\in [0,1]$.

\medskip

In what follows, $\Bbb H(R)$ denotes the hyperbolic disk centered at
the origin with the hyperbolic radius $R=\log\, (1+r)/(1-r)$, $\,
r\in(0,1)$ is its Euclidean radius:
\begin{equation}\label{eq5.5Cf} {\Bbb H}(R)\ =\ \{\ z\in{\Bbb C}:\
h(z,0)\ <\ R\ \}\ ,\ \ \ \ R\in(0,\infty)\ .\end{equation} Further
we also use the notation of the {\bf hyperbolic sine}: $ \sinh t\
:=\ ({e^t-e^{-t}})/{2}\ . $

The following statement is an analog of Lemma 3.1 in \cite{RSY}
adopted to the hyperbolic geometry in the unit disk $\Bbb D:=\{ z\in
\Bbb C: |z|< 1\}$.

\begin{lemma} \label{lem5.5C} {\it\,  Let  $Q:{\Bbb H}(\varepsilon)\rightarrow [0,\infty
]$, $\varepsilon\in(0,1)$, be a measurable function and
$\Phi:[0,\infty ]\rightarrow (0,\infty ]$ be a non-decreasing convex
function with a finite mean integral value $M(\varepsilon)$ of the
function $\Phi\circ Q$ on $\Bbb H(\varepsilon)$. Then
\begin{equation}\label{eq3.222} \int\limits_{0}^{\varepsilon}\
\frac{d\rho}{\rho q(\rho)}\ \ge\ \frac{1}{2}
\int\limits_{\delta(\varepsilon)}^{\infty}\ \frac{d\tau}{\tau
\left[\Phi^{-1}(\tau )\right]}
\end{equation} where $q(\rho)$ is the average of $Q$ on the circle
$\Bbb S(\rho)=\{ z\in \Bbb D: h(z,0)=\rho\}$ and
\begin{equation}\label{eqDELTA}
\delta(\varepsilon)\ =\
\exp\left(4\sinh^2\frac{\varepsilon}{2}\right)\cdot
\frac{M(\varepsilon)}{\varepsilon^2}\ >\ \tau_0\ :=\ \Phi(0)\
>\ 0 \ .
\end{equation}}\end{lemma}

\medskip

\begin{proof} Since $M(\varepsilon)<\infty$ we may assume with no loss of generality that
$\Phi(t)<\infty$ for all $t\in[0,\infty)$ because in the contrary
case $Q\in L^{\infty}$ and then the left-hand side in
(\ref{eq3.222}) is equal to $\infty$. Moreover, we may assume that
$\Phi(t)$ is not constant because in the contrary case
$\Phi^{-1}(\tau)\equiv \infty$ for all $\tau > \tau_0$ and hence the
right-hand side in (\ref{eq3.222}) is equal to 0. Note also that
$\Phi(\tau)$ is (strictly) increasing, convex and continuous in the
segment $[t_*,\infty]$ and
\begin{equation}\label{eqCONSTANT} \Phi(t)\ \equiv\ \tau_0\ \ \ \ \ \ \ \ \ \forall\ t\
\in [0,t_*] \ \ \ \hbox{where}\ \ \ t_*\ :=\
\sup\limits_{\Phi(t)=\tau_0}\ t\ .\end{equation} Setting $H(t)\colon
=\log\Phi(t),$ we see that $H^{-1}(\eta)=\Phi^{-1}(e^{\eta}),$
$\Phi^{-1}(\tau)= H^{-1}(\log\ \tau)$. Thus, we obtain that
\begin{equation}\label{eq5.555I}q(\rho) =
H^{-1}\left(\log \frac{h(\rho)}{\rho^2}\right) = H^{-1}\left(2\log
\frac{1}{\rho} + \log\ h(\rho)\right)\ \ \ \ \ \ \ \ \forall\ \rho\
\in R_*\end{equation} where $h(\rho) :=\ \rho^2\Phi(q(\rho))$ and
$R_* = \{\ \rho\in(0, \varepsilon):\ q(\rho)\
>\ t_*\}$. Then also \begin{equation}\label{eq5.555K}q(e^{-s})\ =\ H^{-1}\left(2s\ +\ \log\
h(e^{-s})\right)\ \ \ \ \ \ \ \ \forall\ s\ \in S_*\end{equation}
where $ S_*\ =\ \{ s\in(\ \log\frac{1}{\varepsilon}\ ,\ \infty\ ):\
q(e^{-s})\
>\ t_*\}$.

Now, by the Jensen inequality, see e.g. Theorem 2.6.2 in \cite{Ra},
we have that
\begin{equation}\label{eq5.555L} \int\limits^{\infty}_{\log \frac{1}{\varepsilon}} h(e^{-s})\
ds\ =\ \int\limits_{0}^{\varepsilon} h(\rho)\ \frac{d\rho}{\rho}\ =\
\int\limits_{0}^{\varepsilon} \Phi(q(\rho))\ \rho\
{d\rho}\end{equation}
$$\le\ \int\limits_{0}^{\varepsilon}\left(\dashint_{S(\rho)}
\Phi(Q(z))\ ds_h(z)\right)\ \rho\ {d\rho}\ \le
2\sinh^2\frac{\varepsilon}{2}\cdot M(\varepsilon)$$ because $\Bbb
H(\varepsilon)$ has the hyperbolic area $A(\varepsilon)=4\pi
\sinh^2\frac{\varepsilon}{2}$ and $\Bbb S(\rho)$ has the hyperbolic
length $L(\rho)=2\pi \sinh\rho$, see e.g. Theorem 7.2.2 in
\cite{Be}, and, moreover, $\sinh \rho \ge \rho$ by the Taylor
expansion. Then arguing by contradiction it is easy to see for the
set $T\ :=\ \{\ s\in (\ \log\frac{1}{\varepsilon}\ ,\ \infty\ ):\
h(e^{-s})\
>\ M(\varepsilon)\ \}$ that its length
\begin{equation}\label{eq5.555N} |T|\ =\ \int\limits_{T}ds\ \le\
2\sinh^2\frac{\varepsilon}{2}\ .
\end{equation}

Next, let us show for $T_*:\ =\ T\cap S_*$ that
\begin{equation}\label{eq5.555O}q\left(e^{-s}\right) \
\le\ H^{-1}\left(2s\ +\ \log\ M(\varepsilon)\right)\ \ \ \ \ \ \ \ \
\ \forall\ s\in\left(\log\frac{1}{\varepsilon}\, ,\infty\right)
\setminus T_*\ .\end{equation} Indeed, note that
$\left(\log\frac{1}{\varepsilon}\, ,\infty\right)\setminus T_* =
\left[\left(\log\frac{1}{\varepsilon}\, ,\infty\right)\setminus
S_*\right] \cup \left[\left(\log\frac{1}{\varepsilon}\,
,\infty\right)\setminus T\right] =
\left[\left(\log\frac{1}{\varepsilon}\, ,\infty\right)\setminus
S_*\right] \cup \left[S_*\setminus T\right]$. The inequality
(\ref{eq5.555O}) holds for $s\in S_*\setminus T$ by (\ref{eq5.555K})
because $H^{-1}$ is a non-decreasing function. Note also that
\begin{equation}\label{eq5.K555}  e^{2s}M(\varepsilon)\ >\ \Phi(0)\ =\ \tau_0
\ \ \ \ \ \ \ \forall\ s\in\left(\log\frac{1}{\varepsilon}\ ,\
\infty\right)\end{equation} and then
\begin{equation}\label{eq5.L555} t_* <
\Phi^{-1}\left(e^{2s}M(\varepsilon) \right) = H^{-1}\left(2s\ +\
\log\ M(\varepsilon)\right)\ \ \ \ \  \forall\ s\in\left(\log
\frac{1}{\varepsilon}\ ,\ \infty\right)\end{equation} Consequently,
 (\ref{eq5.555O}) holds for all $s\in(\ \log\frac{1}{\varepsilon}\ ,\ \infty\ )\setminus S_*$, too.

Since $H^{-1}$ is non-decreasing, we have by
(\ref{eq5.555N})-(\ref{eq5.555O}) that, for  $ \Delta :=\log
{M(\varepsilon)}$,
\begin{equation}\label{eq5.555P} \int\limits_{0}^{\varepsilon}\
\frac{d\rho}{\rho q(\rho)}\ =\ \int\limits^{\infty}_{\log
\frac{1}{\varepsilon}}\ \frac{ds}{q(e^{-s})}\ \ge\
\int\limits_{\left(\log\frac{1}{\varepsilon},\infty\right)\setminus
T_*}\ \frac{ds}{H^{-1}(2s + \Delta)}\ \ge\ \end{equation}
$$
\int\limits^{\infty}_{|T_*|+\log \frac{1}{\varepsilon}}\
\frac{ds}{H^{-1}(2s + \Delta)}\ \ge\
\int\limits^{\infty}_{2\sinh^2\frac{\varepsilon}{2}+\log
\frac{1}{\varepsilon}}\ \frac{ds}{H^{-1}(2s + \Delta)}\ =\
\frac{1}{2} \int\limits_{4\sinh^2\frac{\varepsilon}{2}+\log
\frac{M(\varepsilon)}{\varepsilon^2}}^{\infty}\
\frac{d\eta}{H^{-1}(\eta)}
$$ and after the replacement of variables $\eta=\log\tau$, $\tau =e^{\eta}$, we come to (\ref{eq3.222}).
\end{proof} $\Box$

\begin{theorem} \label{th6}{\it\, Let $Q:{\Bbb H}(\varepsilon)\rightarrow [0,\infty
]$, $\varepsilon\in(0,1)$, be a measurable function such that
\begin{equation}\label{eq5.555} \int\limits_{{\Bbb H}(\varepsilon)}
\Phi (Q(z))\ dh(z)\  <\ \infty\end{equation} where  $\Phi:[0,\infty
]\rightarrow [0,\infty ]$ is a non-decreasing convex function with
\begin{equation}\label{eq3.333a} \int\limits_{\delta_0}^{\infty}\
\frac{d\tau}{\tau \Phi^{-1}(\tau )}\ =\ \infty
\end{equation} for some $\delta_0\ >\ \tau_0\ :=\ \Phi(0).$
Then \begin{equation}\label{eq3.333A} \int\limits_{0}^{\varepsilon}\
\frac{d\rho}{\rho q(\rho)}\ =\ \infty\,,
\end{equation} where $q(\rho)$ is the average of $Q$
on the hyperbolic circle    $h(z,0)=\rho$.}
\end{theorem}

\begin{proof}
If $\Phi(0)\ne 0,$ then Theorem \ref{th6} directly follows from
Lemma \ref{lem5.5C} because $\Phi^{-1}$ is strictly increasing on
the interval $(\tau_0,\infty)$ and $\Phi^{-1}(\delta_0)>0$. In the
case $\Phi(0)=0,$ let us fix a number $\delta\in (0, \delta_0)$ and
set $\Phi_*(t)=\Phi(t),$ if $\Phi(t)>\delta,$ and
$\Phi_*(t)=\delta,$ if $\Phi(t)\le \delta.$ Then by (\ref{eq5.555})
we have that $\int\limits_{{\Bbb H}(\varepsilon)}\Phi_*(Q(z))\
dh(z)<\infty$ because $|\Phi_*(t)-\Phi(t)|\le \delta$ and the
measure of ${\Bbb H}(\varepsilon)$ is finite. Moreover,
$\Phi_*^{-1}(\tau)=\Phi^{-1}(\tau)$ for $\tau\ge\delta$ and then by
(\ref{eq3.333a})
$\int\limits_{\delta_0}^{\infty}\frac{d\tau}{\tau\Phi_*^{-1}(\tau)}=\infty.$
Thus, (\ref{eq3.333A}) holds again by Lemma \ref{lem5.5C}.
\end{proof}$\Box$

\bigskip

\begin{remark}\label{rmk4.7www} {\rm Note that condition (\ref{eq3.333a}) implies that
\begin{equation}\label{eq3.a333} \int\limits_{\delta}^{\infty}\
\frac{d\tau}{\tau \Phi^{-1}(\tau )}\ =\ \infty\ \ \ \ \ \ \ \ \ \
\forall\ \delta\ \in\ [0,\infty)\,.   \end{equation} but relation
(\ref{eq3.a333}) for some $\delta\in[0,\infty),$ generally speaking,
does not imply (\ref{eq3.333a}). Indeed, (\ref{eq3.333a}) evidently
implies  (\ref{eq3.a333}) for $\delta\in [0,\delta_0),$ and, for
$\delta\in(\delta_0,\infty),$ we have that
\begin{equation}\label{eq3.e333} 0\ \le\
\int\limits_{\delta_0}^{\delta}\ \frac{d\tau}{\tau \Phi^{-1}(\tau
)}\ \le\ \frac{1}{\Phi^{-1}(\delta_0)}\ \log\
\frac{\delta}{\delta_0}\ <\ \infty
\end{equation} because the function $\Phi^{-1}$ is non-decreasing and
$\Phi^{-1}(\delta_0)>0.$ Moreover, by the definition of the inverse
function $\Phi^{-1}(\tau)\equiv 0$ for all $\tau \in [0,\tau_0],$
$\tau_0=\Phi(0)$, and hence (\ref{eq3.a333}) for
$\delta\in[0,\tau_0),$ generally speaking, does not imply
(\ref{eq3.333a}). If $\tau_0 > 0$, then
\begin{equation}\label{eq3.c333} \int\limits_{\delta}^{\tau_0}\
\frac{d\tau}{\tau \Phi^{-1}(\tau )}\ =\ \infty\ \ \ \ \ \ \ \ \ \
\forall\ \delta\ \in\ [0,\tau_0)  \end{equation} However, relation
(\ref{eq3.c333}) gives no information on the function $Q$ itself
and, consequently, (\ref{eq3.a333}) for $\delta < \Phi(0)$ cannot
imply (\ref{eq3.333A}) at all.}
\end{remark}

\bigskip

\section{Other criteria for homeomorphic extension in prime ends}

Theorem \ref{th2} has a magnitude of other consequences thanking to
Theorem \ref{th6}.


\begin{theorem}\label{th5} {\it Under the hypothesis of
Theorem \ref{th1}, let
\begin{equation}\label{eqOSKRSS10.36b} \int\limits_{D(p_0,\varepsilon_0)}\Phi_{p_0}\left(K_{f}(p)\right)\ dh(p)\ <\ \infty
\qquad \forall\ p_0\in\partial D\end{equation} for
$\varepsilon_0=\varepsilon(p_0)$ and a nondecreasing convex function
$\Phi_{p_0}:[0,\infty)\to[0,\infty)$ with
\begin{equation}\label{eqOSKRSS10.37b}
\int\limits_{\delta(p_0)}^{\infty}\frac{d\tau}{\tau\Phi_{p_0}^{-1}(\tau)}\
=\ \infty\end{equation} for $\delta(p_0)>\Phi_{p_0}(0)$. Then $f$ is
extended to a homeomorphism of $\overline{D}_P$ onto
$\overline{D'}_P$.}
\end{theorem}

\begin{proof} Indeed, in the case of the hyperbolic Riemann
surfaces, (\ref{eqOSKRSS10.36b}) and (\ref{eqOSKRSS10.37b}) imply
(\ref{e:6.4dc}) by Theorem \ref{th6} and, after this, Theorem
\ref{th5} becomes a direct consequence of Theorem \ref{th2}. In the
more simple case of the elliptic and parabolic Riemann surfaces, we
similarly can apply Theorem 3.1 in \cite{RSY} for the Euclidean
plane instead of Theorem \ref{th6}. \end{proof}$\Box$

\begin{corollary}\label{cor4} {\it In particular, the conclusion of Theorem
\ref{th5} holds if
\begin{equation}\label{eqOSKRSS6.6.6}
\int\limits_{D(p_0,\varepsilon_0)}e^{\alpha_0 K_{f}(p)}\ dh(p)\ <\
\infty\qquad \forall\ p_0\in\partial D\end{equation} for some
$\varepsilon_0=\varepsilon(p_0)>0$ and $\alpha_0=\alpha(p_0)>0$.}
\end{corollary}


\begin{remark}\label{rmOSKRSS200000}
{\rm Note that by Theorem 5.1 and Remark 5.1 in \cite{KR2} condition
(\ref{eqOSKRSS10.37b}) is not only sufficient but also necessary for
a continuous extendibility to the boundary of all mappings $f$ with
the integral restriction (\ref{eqOSKRSS10.36b}).

Note also that by Theorem 2.1 in \cite{RSY}, see also Proposition
2.3 in \cite{RS1}, (\ref{eqOSKRSS10.37b}) is equivalent to every of
the conditions from the following series:
\begin{equation}\label{eq333Y}\int\limits_{\delta(p_0)}^{\infty}
H'_{p_0}(t)\ \frac{dt}{t}=\infty\ ,\quad\ \delta(p_0)>0\
,\end{equation}
\begin{equation}\label{eq333F}\int\limits_{\delta(p_0)}^{\infty}
\frac{dH_{p_0}(t)}{t}=\infty\ ,\quad\ \delta(p_0)>0\ ,\end{equation}
\begin{equation}\label{eq333B}
\int\limits_{\delta(p_0)}^{\infty}H_{p_0}(t)\ \frac{dt}{t^2}=\infty\
,\quad\ \delta(p_0)>0\ ,
\end{equation}
\begin{equation}\label{eq333C}
\int\limits_{0}^{\Delta(p_0)}H_{p_0}\left(\frac{1}{t}\right)\,dt=\infty\
,\quad\ \Delta(p_0)>0\ ,
\end{equation}
\begin{equation}\label{eq333D}
\int\limits_{\delta_*(p_0)}^{\infty}\frac{d\eta}{H_{p_0}^{-1}(\eta)}=\infty\
,\quad\ \delta_*(p_0)>H_{p_0}(0)\ ,
\end{equation}
where
\begin{equation}\label{eq333E}
H_{p_0}(t)=\log\Phi_{p_0}(t)\ .\end{equation}

Here the integral in (\ref{eq333F}) is understood as the
Lebesgue--Stieltjes integral and the integrals in (\ref{eq333Y}) and
(\ref{eq333B})--(\ref{eq333D}) as the ordinary Lebesgue integrals.


It is necessary to give one more explanation. From the right hand
sides in the conditions (\ref{eq333Y})--(\ref{eq333D}) we have in
mind $+\infty$. If $\Phi_{p_0}(t)=0$ for $t\in[0,t_*(p_0)]$, then
$H_{p_0}(t)=-\infty$ for $t\in[0,t_*(p_0)]$ and we complete the
definition $H'_{p_0}(t)=0$ for $t\in[0,t_*(p_0)]$. Note, the
conditions (\ref{eq333F}) and (\ref{eq333B}) exclude that $t_*(p_0)$
belongs to the interval of integrability because in the contrary
case the left hand sides in (\ref{eq333F}) and (\ref{eq333B}) are
either equal to $-\infty$ or indeterminate. Hence we may assume in
(\ref{eq333Y})--(\ref{eq333C}) that $\delta(p_0)>t_0$,
correspondingly, $\Delta(p_0)<1/t(p_0)$ where
$t(p_0):=\sup\limits_{\Phi_{p_0}(t)=0}t$, set $t(p_0)=0$ if
$\Phi_{p_0}(0)>0$.


The most interesting among the above conditions is (\ref{eq333B}),
i.e. the condition:
\begin{equation}\label{eq5!}
\int\limits_{\delta(p_0)}^{\infty}\log \Phi_{p_0}(t)\ \
\frac{dt}{t^{2}}\ =\ +\infty\ \ \ \ \ \ \ \ \hbox{for some}\ \ \
\delta(p_0)\ >\ 0\ .
\end{equation}

Finally, it is necessary to note that the restriction on
nondegeneracy of boundary components of domains in Theorem \ref{th1}
as well as in all other theorems is not essential because this
simplest case is included in our previous paper \cite{RV3}. }
\end{remark}

\noindent
{\bf Vladimir Ryazanov, Sergei Volkov,}\\
Institute of Applied Mathematics and Mechanics,\\
National Academy of Sciences of Ukraine,\\
84116, Ukraine, Slavyansk, 19 General Batyuk Str.,\\
vl.ryazanov1@gmail.com, sergey.v.volkov@mail.ru

\end{document}